\documentclass[12pt]{amsart}
\usepackage{color, hyperref}
\usepackage{enumerate}

\makeatletter

\newcommand{\Rmnum}[1]{\expandafter\@slowromancap\romannumeral #1@}
\makeatother

\theoremstyle{plain}
\newtheorem{theorem} {Theorem} [section]
\newtheorem{lemma} [theorem]{Lemma}
\newtheorem{proposition}[theorem]{Proposition}
\newtheorem{corollary} [theorem]{Corollary}
\theoremstyle{definition}
\newtheorem{definition} [theorem] {Definition}
\newtheorem{example}[theorem] {Example}

\numberwithin{equation}{section}

\title[totally geodesic discs]{Holomorphicity of totally geodesic Kobayashi isometry between bounded symmetric domains}
\author[S.-Y. Kim, A. Seo ]{Sung-Yeon Kim,  Aeryeong Seo}
\address{S.-Y. Kim: Center for Complex Geometry, Institutue for Basic Science, 
55, Expo-ro, Yuseong-gu, Daejeon, Korea, 34126 }
\email{sykim8787@ibs.re.kr}

\address{A. Seo: Department of Mathematics,
Kyungpook National University,
Daegu 41566, Republic of Korea}%
\email{aeryeong.seo@knu.ac.kr}

\subjclass[2010]{32F45, 32Q45, 32M15, 53C35}
\keywords{bounded symmetric domain, Kobayashi metric, totally geodesic isometric embedding, holomorphicity}

\begin{document}
\maketitle

\markboth{ }{}

\begin{abstract}
In this paper, we study the holomorphicity of totally geodesic Kobayashi isometric embeddings between bounded symmetric domains. 
First we show that for a $C^1$-smooth totally geodesic Kobayashi isometric embedding $f\colon \Omega\to\Omega'$ where $\Omega$, $\Omega'$ 
are bounded symmetric domains, if $\Omega$ is irreducible and $\text{rank}(\Omega) \geq \text{rank}(\Omega')$ or more generally, $\text{rank}(\Omega) \geq \text{rank}(f_*v)$ for any tangent vector $v$ of $\Omega$, then $f$ is either holomorphic or anti-holomorphic.
Secondly we characterize $C^1$ Kobayashi isometries from a reducible bounded symmetric domain to itself. 
\end{abstract}
\section{Introduction}

In the geometric theory of holomorphic functions on complex manifolds, the
concept of invariant metrics plays an important role. 
Invariant metrics refer to Finsler or Hermitian metrics which are invariant with respect to biholomorphic mappings. 
Within the theory of several complex variables the two most important examples of such objects are the Bergman and Kobayashi metrics.
In general for a complex manifold and its invariant metric, the isometry group properly contains the set of holomorphic and anti-holomorphic diffeomorphisms. However, these two coincide under certain conditions, for example if one considers the Bergman metrics on $C^\infty$ strongly pseudoconvex domains in the complex Euclidean spaces (\cite[Theorem 1.17]{Greene_Krantz_1982}).

It is natural to consider the question, under which circumstances any isometry between complex manifolds equipped with their invariant metrics is either holomorphic or anti-holomorphic.
The goal of this paper is to investigate the holomorphicity of the totally geodesic isometric embeddings in the case when complex manifolds are bounded symmetric domains equipped with Kobayashi metrics.

There are a few works on this direction. 
Let $\Omega\subset \mathbb C^n$, $\Omega'\subset \mathbb C^m$ be bounded domains.
Let $f\colon \Omega\rightarrow \Omega'$ be an isometric embedding with respect to the Kobayashi or Carath\'eodory metrics.
By using Pinchuk's scaling method  Seshadri-Verma \cite{Seshadri_Verma_2006} proved that if $\Omega$ and $\Omega'
$ are strongly pseudoconvex domains with $m=n$ and $f$ extends $C^1$-smoothly to the boundary, then $f$ should be biholomorphic or anti-biholomorphic. 
If $\Omega$ and $\Omega'$ are $C^3$-smooth strongly convex domains, then 
Gaussier-Seshadri \cite{Gaussier_Seshadri_2013} showed that any totally geodesic disc is either holomorphic or anti-holomorphic. Moreover they proved that $f$ has the same property as well. 
More recently  Antonakoudis \cite{Antonakoudis_2017} obtained corresponding result when $\Omega$ and  $\Omega'$ are complete disc-rigid domains which include strongly convex bounded domains and Teichm\"uller spaces of finite dimension. 

One common ingredient in the proofs of these results is the unique existence of complex geodesic containing given two points in $\Omega$ or in the direction of given vector on $\Omega$.
For strongly convex domains,  it is due to a work of Lempert \cite{Lempert_1981, Lempert_1982}
and for Teichm\"uller spaces, it is a classical result (see \cite{Earle_Kra_Krushkal_1994}).

For a bounded symmetric domain $\Omega$ with rank at least two, 
as Antonakoudis pointed out in \cite{Antonakoudis_2017}, there exists a totally geodesic disc $f\colon \Delta \rightarrow \Omega$ which is neither holomorphic nor anti-holomorphic.
Such disc is given by
$$
f(z) = (z,\bar z)\in \Delta^2\subset \Omega,
$$
where $\Delta^2\subset \Omega$ is a totally geodesic bidisc in $\Omega$ with respect to the Bergman metric.
This misfortune leads us to examine how totally geodesic isometric embeddings into bounded symmetric domains behave.

Motivated by the work of Tsai(\cite{Tsai_1993}) about the rigidity of proper holomorphic mappings between bounded symmetric domains, 
we obtain the following theorem.

\begin{theorem}\label{main}
Let $\Omega$ and $\Omega'$ be bounded symmetric domains and let $f\colon\Omega\to \Omega'$ be a $C^1$-smooth totally geodesic isometric embedding with respect to their Kobayashi metrics. Suppose that $\Omega$ is irreducible. Suppose further that
$${\rm rank }(\Omega)\geq {\rm rank }(\Omega').$$
Then 
$$ {\rm rank }(\Omega)={\rm rank }(\Omega')$$
and $f$ is either holomorphic or anti-holomorphic.
\end{theorem}

For reducible bounded symmetric domains, we have the following result:
\begin{theorem}\label{Kobayashi isometry}
Let $\Omega=\Omega_1\times\cdots\times\Omega_m$ be a bounded symmetric domain, where each $\Omega_i$ is irreducible. 
Then, up to permutation of irreducible factors, any $C^1$ Kobayashi isometry $F\colon\Omega\to \Omega$ is of the form 
$$F(z_1,\ldots,z_m)=(f_1(z_1),\cdots,f_m(z_m)) \quad z_i\in \Omega_i,$$
where each $f_i\colon\Omega_i\to \Omega_i$ is either biholomorphic or anti-biholomorphic.
\end{theorem}

Note that an analogous theorem of Theorem \ref{Kobayashi isometry} for proper holomorphic maps was proved by Seo \cite{Seo_2018}.

For a bounded symmetric domain $\Omega$, let $\widehat \Omega$ denote its compact dual. On $\widehat \Omega$, there is a space of minimal rational curves $C$ which are homologically generators of $H_2(\widehat \Omega, \mathbb Z)$. The intersection of $C$ with $\Omega$ is called a {\it minimal disc}, and the tangent vectors of minimal discs are called {\it rank one vectors}.
We will say a holomorphic disc into a bounded symmetric domain is of rank one if it is tangential to rank one vectors at all points. Definition of the rank of vectors is given in 2. Preliminaries.

\begin{theorem}\label{main-2}
Let $\Omega$ and $\Omega'$ be bounded symmetric domains and let $f\colon\Omega\to \Omega'$ be a $C^1$-smooth totally geodesic isometric embedding that extends continuously to the boundary with respect to their Kobayashi metrics. Suppose $\Omega$ is irreducible. Suppose further that
\begin{equation}\label{vector-rank}
{\rm rank }(\Omega)\geq {\rm rank}~f_*(v)
\end{equation}
for all $v\in T\Omega$.
Then $f$ is either holomorphic or anti-holomorphic.
\end{theorem}
\begin{corollary}
Any rank one totally geodesic Poincar\'{e} disc in a bounded symmetric domain which is continuous up to the boundary is a minimal disc.
\end{corollary}
We remark that there are a lot of rank one holomorphic discs which are not minimal discs (see \cite{Griffiths_Harris_1979, Choe_Hong_2004}).

Theorem \ref{main} and Theorem \ref{main-2} are not true anymore if there are no rank conditions:
for $p,q,p',q'\in \mathbb N$ with $p<p'$, $q<q'$, $p' \leq q$, $p\leq q$, let $\Omega=\Omega^I_{p,q}$ and $\Omega' = \Omega^I_{p', q'}$ be irreducible bounded symmetric domains of type I:
$$
\Omega^I_{p,q} := \left\{ Z \in M(p,q;\mathbb C) : I - \overline Z^t Z>0 \right\}.
$$
Then for any anti-holomorphic map $\varphi\colon \Omega^I_{p,q} \to \Omega^I_{p'-p, q'-q} $, the map 
$$
Z\mapsto \left(
\begin{array}{cc}
Z &0\\
0& \varphi(Z)
\end{array}\right)
\colon \Omega^I_{p,q}\to \Omega^I_{p', q'}
$$
is a totally geodesic isometric embedding with respect to the Kobayashi metrics.

\bigskip

{\bf Acknowledgement} 
The first author was supported by the Institute for Basic Science (IBS-R032-D1-2021-a00). The second author was supported by Basic Science Research Program through the National Research Foundation
of Korea (NRF) funded by the Ministry of Education (NRF-2019R1F1A1060175).

\section{Preliminaries}
\subsection{Bounded symmetric domains}
In this section, we recall some facts about bounded symmetric domains. For more details, see \cite{Mok89, Mok_Tsai_1992, Mok_2016}.

A bounded domain $\Omega$ in the complex Euclidean space is called symmetric if for each $p\in \Omega$, there exists a holomorphic automorphism $I_p$ such that $I_p^2$ is the identity map of $\Omega$ which has $p$ as an isolated fixed point. 
Bounded symmetric domains are homogeneous complex manifolds and their Bergman metrics are K\"ahler-Einstein with negative holomorphic sectional curvatures.
It is well known that all Hermitian symmetric spaces of non-compact type can be realized as convex bounded symmetric domains by the Harish-Chandra realizations.
Moreover there exists a one to one correspondence between the set of Hermitian symmetric spaces of the non-compact type and the compact type. 
For a bounded symmetric domain $\Omega$, the corresponding Hermitian symmetric space of the compact type is called its compact dual.
There exists a canonical embedding, which is called the Borel embedding, from $\Omega$ to its compact dual.

In what follows $M(p,q; \mathbb C)$ denotes the set of $p\times q$ matrices with complex coefficients.
The set of irreducible Hermitian symmetric spaces of non-compact type consists of four classical types and two exceptional types. We list the irreducible bounded symmetric domains which are the Harish-Chandra realizations of them as follows;
$$
\Omega^I_{p,q} := \left\{ Z \in M(p,q;\mathbb C) : I - \overline Z^t Z>0 \right\}, \quad p,q \geq 1;
$$
$$\Omega^{II}_n := \left\{ Z\in \Omega^I_{n,n} : Z^t = -Z \right\}, \quad n\geq 2;
$$
$$
\Omega^{III}_n := \left\{ Z\in \Omega^I_{n,n} : Z^t = Z \right\}, \quad n\geq 1;
$$
$$
\Omega^{IV}_n:= \left\{ (z_1, \ldots, z_n) \in \mathbb C^n : \| z\|^2 < 2, \| z\|^2 <1 + \bigg| \frac{1}{2} \sum_{j=1}^n z_j^2 \bigg|^2 \right\}, \quad n\geq 3;
$$
$$ 
\Omega_{16}^V = \left\{z\in M_{1,2}^{\mathbb O_\mathbb C} : 
1-(z|z) + (z^\#|z^\#) >0,\, 2-(z|z)>0 \right\};
$$
\begin{equation}\nonumber
\begin{aligned}
\Omega_{27}^{VI} &= 
\left\{ z\in H_3(\mathbb O_\mathbb C)  : 
1-(z|z) + (z^\#|z^\#) - |\det z|^2 >0,\right. \\
 &\quad\quad\quad\quad\quad\quad\quad\quad
\quad\quad\quad
\left. 3-2(z|z) + (z^\#|z^\#) >0, 3-(z|z)>0 \right\}.
\end{aligned}
\end{equation}
Here, $\mathbb O_\mathbb C$ is the complex 8-dimensional algebra of complex octonions. For the notation $M_{1,2}^{\mathbb O_{\mathbb C}}$, $H_3(\mathbb O_\mathbb C)$, $z^\#$ for the exceptional type domains, see \cite{Roos_2008}.

\begin{theorem}[Polydisc Theorem]\label{polydisc theorem}
Let $\Omega$ be a bounded symmetric domain and $b_\Omega$ be its Bergman metric. Let $X$ be the compact dual of $\Omega$ and $g_c$ be its K\"ahler-Einstein metric.
There exists a totally geodesic complex submanifold $D$ of $(\Omega, b_\Omega)$ such that 
$(D, b_\Omega|_D)$ is holomorphically isometric to a Poincar\'e polydisc $(\Delta^r, \rho)$ and $$
\Omega = \bigcup_{\gamma\in K} \gamma D
$$
where $K$ denotes an isotropy subgroup of $\text{Aut}(\Omega)$.
Moreover, there exists a totally geodesic complex submanifold $S$ of $(X,g_c)$
containing $D$ as an open subset such that $(S, g_c|_S)$ is isometric to a polysphere 
$((\mathbb P^1)^r, \rho_c)$ equipped with a product Fubini-Study metric $\rho_c$.
\end{theorem}
The dimension of $D$ in Theorem \ref{polydisc theorem} is called the {\it rank} of $\Omega$. 
A projective line $C\cong \mathbb P^1$ in $X$ which is a homological generator of $H_2(X,\mathbb Z)$ is called a minimal rational curve of $X$ and it is totally geodesic on $(X, g_c)$. 
Via the Harish-Chandra and Borel embedding, we call the intersection of $C$ with $\Omega$ a {\it minimal disc}.
In particular, minimal discs $\Delta$ can be precisely expressed by  $\{ (z,0,\ldots,0): |z|<1\}\subset \Delta^r \cong D$ where $D$ is the totally geodesic polydisc in Theorem \ref{polydisc theorem},
and $\{(z,0\ldots, 0)\in (\mathbb P^1)^r\cong S\}\cong \mathbb P^1\subset X$ is the minimal rational curve on $X$ such that $\mathbb P^1 \cap \Omega = \Delta$.

To each totally geodesic polydisc $\Delta^k$ with $1\leq k\leq r$ in $\Omega$, there exists a bounded symmetric subdomain $(\Delta^k)^\perp$ of rank $r-k$ in $\Omega$ such that $\Delta^k\times (\Delta^k)^\perp$ 
can be embedded in $\Omega$ as a totally geodesic submanifold with respect to the Bergman metrics.
For each $b\in (\partial \Delta)^k$, $\{b\}\times (\Delta^k)^\perp$ is a boundary component of $\Omega$, i.e. it is a maximal complex submanifold contained in $\partial \Omega$. 
Moreover for any boundary component $B$ of $\Omega$, there exists a totally geodesic polydisc $\Delta^k$ and a point $b\in (\partial\Delta)^k$ so that $B\cong \{b\}\times (\Delta^k)^\perp$.
For each irreducible bounded symmetric domain $\Omega$, $(\Delta^k)^\perp$ is given by Table \ref{characteristic subdomains} below for each $k$ with $1\leq k\leq \text{rank}(\Omega)$ and the canonical embedding into $\Omega$.

\begin{table}[ht]\caption{Characteristic subdomains }
\label{characteristic subdomains}
\begin{tabular}{c|c|c|c|c|c|c}
$\Omega$& $\Omega_{p,q}^I (p\leq q)$& $\Omega^{II}_n$& $\Omega^{III}_n$& $\Omega^{IV}_n$ & $\Omega^V_{16}$&$\Omega_{27}^{VI}$  \\[4pt]\hline 
&& &&&&\\[-8pt]
$(\Delta^k)^\perp$ & $\Omega^I_{p-k, q-k}$ & $\Omega^{II}_{n-2k}$& $\Omega_{n-k}^{III}$ & $\Delta(k=1)$ & $\mathbb B^5(k=1)$ & $\Omega_8^{IV} (k=2)$,
$\Delta(k=1)$
\end{tabular}
\end{table}
Here $\mathbb B^n:=\{z\in \mathbb C^n: |z|<1\}$ denotes the $n$-dimensional unit ball.

Denote by $T\Omega$ the holomorphic tangent bundle of $\Omega$.
For a point $x\in \Omega$, let $v\in T_x\Omega$ be a unit vector. 
If $v$ realizes the minimum of holomorphic sectional curvature of $b_\Omega$, 
then we call $v$ a rank $1$-vector (or {\it characteristic vector} in \cite{Mok89}). 
A vector $v\in T_x\Omega$ is a rank $1$-vector if and only if there exists a 
minimal disc $\Delta\subset \Omega$ such that $v\in T_x\Delta$.
For any $v\in T_x\Omega$, there exists a unique totally geodesic polydisc $\Delta^k\subset \Omega$ with minimum dimension such that $v\in T_x\Delta^k$ and we will call $v$ a rank $k$-vector.
Let $\mathcal C_x^k$ denote the set
$$
\mathcal C_x^k := \{ [v] : v \in T_x\Omega\text{ is a rank } k\text{-vector at } x\}
$$ 
in $\mathbb PT_x\Omega$.
The $k$-th characteristic bundle over $\Omega$ is defined by $$
\mathcal C^k (\Omega) := \bigcup_{x\in \Omega} \mathcal C^k_x \subset \mathbb PT\Omega.$$
We remark that $\mathcal C^k_x$ is parallel with $\mathcal C_0^k$ for any $k$ and $x\in \Omega$ in Harish-Chandra coordinates (\cite{Mok_Tsai_1992}). We will compare two vectors $v\in T_p\Omega$ and $w\in T_q\Omega$ with different $p,q \in \Omega$ using parallel translation with respect to Harish-Chandra coordinates.
For convenience, we will denote $\mathcal C^1_x$, $\mathcal C^1(\Omega)$ by $\mathcal C_x$, $\mathcal C(\Omega)$, respectively.

\begin{definition}
A smooth real submanifold $N\subset \Omega$ is called an {\it integral manifold} of $\mathcal {RC}^k(\Omega)$, if 
for any $v\in TN$, there exists $[w]\in \mathcal C^k(\Omega)$ with $w\in T\Omega$ such that $v=\text{Re}(w)$.
We will say that a map $f\colon N\rightarrow \Omega$ is tangential to $\mathcal{RC}^k(\Omega)$ if $f(N)$ is an integral manifold of $\mathcal{RC}^k(\Omega)$.
\end{definition}

Let $v$ be a unit rank 1-vector at $x\in \Omega$ and write 
$R_v$ for the Hermitian bilinear form on $T_x\Omega$ defined by 
$R_v(\xi, \eta) := R(v, \bar v , \xi, \bar \eta)$
where $R$ denotes the curvature tensor of $b_\Omega$. Then 
the eigenspace decomposition of 
$T_xX$ with respect to $R_v$ is given by 
$$
T_x \Omega = \mathbb Cv + \mathcal H_v + \mathcal N_v
$$
corresponding to the eigenvalues $2$, $1$ and $0$ respectively.
For $[v]\in \mathcal C_0$,
let $\Delta\subset \Omega$ be a minimal disc so that $v\in T_0\Delta$. Then we have 
\begin{equation}\nonumber
\mathcal N_v = T_0 \Delta^\perp.
\end{equation}
Identifying $T_{[v]} (\mathbb P T_0\Omega)$ with $T_0\Omega/\mathbb Cv$, we have $T_{[v]} \mathcal C_0 \cong (\mathbb C v + \mathcal H_v)/\mathbb C v$.

Note that any vector $v\in T\Omega$ can be expressed as a linear combination of rank one vectors by Polydisc Theorem.
For a nonzero vector $v=\sum_j c_jv_j$ with rank one vectors $v_j$, define
$$\mathcal N_{[v]}:=\bigcap_j\mathcal N_{[v_j]}.$$
Since what matters is the directions of rank one vectors $v_j$, for a totally geodesic polydisc $\Delta^s\subset \Omega$, we denote $\mathcal N_{[v]}$ by $\mathcal N_{[T_x \Delta^s]}$ when $v$ is the vector realizing the maximal holomorphic sectional curvature of $\Delta^s$ with respect to the Bergman metric.
For a totally geodesic polydisc $\Delta^k\subset \Omega$, we have
$$T_p (\Delta^k)^\perp=\bigcap_{[v]\in \mathbb PT_p\Delta^k} \mathcal N_{[v]}.$$ 
We remark that $\mathcal N_{[v]}$ is parallel along totally geodesic holomorphic disc $\Delta$ which is tangent to $v$.

\begin{lemma}[Lemma 3.6 in \cite{Mok_Tsai_1992}]\label{2-ball}
Let $u\in \mathcal H_v$ be a root vector of unit norm. 
Then either 
\begin{enumerate}
\item
$R_{u\bar  u u\bar u}=R_{v\bar v v \bar v}$ and $(\mathbb C v +\mathbb C u)\cap \Omega\cong \mathbb B^2$ is totally geodesic in $(\Omega, b_\Omega)$, or 
\item 
$R_{u\bar u u\bar u}=\frac{1}{2}R_{v\bar v v \bar v}$ and there exists $w\in \mathcal N_v$ such that 
$(\mathbb C v + \mathbb C u + \mathbb C w)\cap \Omega \cong \Omega_3^{IV}$ is totally geodesic in $(\Omega, b_\Omega)$.
\end{enumerate}
\end{lemma}
Let $\mu$, $\phi$ be roots of $v$, $u$ respectively. 
If $R_{u\bar  u u\bar u}=\frac{1}{2}R_{v\bar v v \bar v}$, then $\mu$, $\mu-\phi$, $\mu-2\phi$ are roots (see \cite[Lemma 3.6]{Mok_Tsai_1992}).

\begin{lemma}\label{ball}
\begin{enumerate}
\item
There are canonical totally geodesic isometric embeddings with respect to the Bergman and the Kobayashi metric
$$
\nu\colon \Omega^{II}_n,\, \Omega^{III}_n\hookrightarrow \Omega^I_{n,n} \quad \text{ and }\quad
\nu\colon \Omega^{IV}_{2k+1}\hookrightarrow \Omega^{IV}_{2k+2}
$$
such that for any minimal disc $\Delta\subset \Omega$, 
$$
(\nu(\Delta) \times \nu(\Delta)^\perp)\cap \nu(\Omega) = \nu(\Delta\times\Delta^\perp)
$$
for $\Omega= \Omega^{II}_n,\, \Omega^{III}_n$ or $\Omega^{IV}_{2k+1}$. 
\item
For $\Omega = \Omega^I_{p,q}$, $\Omega^{IV}_{2k+2}$,  $\Omega^V_{16}$ 
or $\Omega^{VI}_{27}$ and  a characteristic vector $v\in T_x\Omega$,
there exists a basis $\{e_1,\ldots, e_\ell\}$ of $\mathcal H_v$ 
such that $R(e_i,\bar e_i, e_i,\bar e_i)=2$ for any $i$. In particular $\Omega\cap \{ \mathbb C v + \mathbb C e_i\}\cong \mathbb B^2$.
\end{enumerate}
\end{lemma}
\begin{proof}
{\bf(1):} Let $\nu$ be the trivial embedding for $
\nu\colon \Omega^{II}_n,\, \Omega^{III}_n\hookrightarrow \Omega^I_{n,n} $ and let $\nu(z_1,\ldots, z_{2k+1} ) = (z_1,\ldots, z_{2k+1}, 0)$ for $\nu\colon \Omega^{IV}_{2k+1}\hookrightarrow \Omega^{IV}_{2k+2}$.

{\bf (2):} By Lemma \ref{2-ball} we only need to show that $R(\xi, \bar \xi, \xi,\bar\xi)=R(v, \bar v, v, \bar v)$ for any root vector $\xi\in \mathcal H_v$.
By the curvature formula for root vectors $e_\mu$ and $ e_\varphi$,
$$
R(e_\mu, \bar e_\mu, e_\varphi, \bar  e_\varphi) = c([e_\mu, e_{-\mu}]; [  e_\varphi,  e_{-\varphi}])
$$
where $(\cdot, \cdot)$ denotes $-B(\cdot, \bar \cdot)$ for the Killing form $B$ of $\mathfrak g^{\mathbb C}$ and a positive constant $c$.

\medskip

For $\Omega = \Omega^I_{p,q}$, the identity component of the automorphism group of $\Omega$ is $SU(p,q)$ and its Lie algebra is $\frak{su}(p,q)$. For a Cartan subalgebra $\frak h:= \text{diag}[a_{11}, \ldots, a_{p+q, p+q}]$, let $\{ \pm (\epsilon_i - \epsilon_j) : 1\leq i<j\leq p+q\}$ be the set of roots of $\frak{su}(p,q)\otimes{\mathbb C} = \frak{sl}(p+q, \mathbb C)$
where $\epsilon_j(h) = a_{jj}$ for $h=(a_{11},\ldots, a_{p+q, p+q})$.  
Let $a_{ij}$ denote the 
root vector of $\epsilon_i - \epsilon_j$. 
Then $a_{ij}$ with $1\leq i \leq p < j\leq p+q$ consist of a basis of the holomorphic tangent space of $\Omega$ at $0$. 
For a characteristic vector $v= a_{1 p+1}$, 
we have $\mathcal H_v= \text{span} \{a_{1j} : p+2\leq j\leq p+q\} \cup \{a_{i p+1} : 2\leq i\leq p\}$ and $\mathcal N_v = \text{span}\{ a_{ij} : 2\leq i\leq p, p+2\leq j\leq p+q\}$.
Hence $R(a_{ij}, \bar a_{ij}, a_{ij}, \bar a_{ij})=R(v, \bar v, v, \bar v)$ for any $i,j$.

\medskip

For $\Omega = \Omega^{IV}_{2k+2}$, the identity component of the automorphism group of $\Omega$ is $G = SO(2k,2)$ and its Lie algebra is $\frak g = \frak{so}(2k,2)$. For a Cartan subalgebra $\frak h:= \text{diag}(a_{11}, \ldots,\\ a_{k+1, k+1}, -a_{11},\ldots, -a_{k+1, k+1})$, let $\{ \pm \epsilon_i \pm \epsilon_j : 1\leq i<j\leq k+1\}$ be the set of roots of $\frak g^{\mathbb C} = \frak{so}(2k+2, \mathbb C)$.
Then the root vectors of 
$\{ \epsilon_i \pm \epsilon_{k+1} :  1\leq i \leq k \}$ consist of a basis of the holomorphic tangent space of $\Omega$ at $0$.
For a characteristic vector $v$ which is a root vector of $\epsilon_k - \epsilon_{k+1}$, $\mathcal H_v$ is the span of root vectors corresponding to $\{\epsilon_i-\epsilon_{k+1} :1\leq i \leq k-1\} $ and $\mathcal N_v$ is the span of root vector of $\epsilon_k + \epsilon_{k+1}$.
Hence $R(\xi, \bar \xi, \xi,\bar\xi)=R(v, \bar v, v, \bar v)$ for any root vector $\xi$ corresponding to $\epsilon_i-\epsilon_{k+1}$ with $1\leq i\leq k-1$.

\medskip

For $\Omega = \Omega^{VI}_{27}$, the identity component of the automorphism group of it is the exceptional simple Lie group $G = E_7$ and its Lie algebra is $\frak g = \frak e_7$. 
The noncompact positive roots of $\frak e_7$ is listed on \cite[page 150]{Drucker_1978}: 
\begin{equation}\nonumber
\begin{array}{c}
\lambda_t , \,
\frac{1}{2}(\lambda_2 + \lambda_3) + 2\varepsilon \rho_s, \, \\
\frac{1}{2}(\lambda_1 + \lambda_3) + \varepsilon (-\rho_0 + \rho_1 + \rho_2 +\rho_3),
\frac{1}{2}(\lambda_1 + \lambda_3) + \varepsilon (-\rho_0 + \rho_i - \rho_j -\rho_k), \,\\
\frac{1}{2}(\lambda_1 + \lambda_2) - \varepsilon (\rho_0 + \rho_1 + \rho_2 +\rho_3),\,
\frac{1}{2}(\lambda_1 + \lambda_2) + \varepsilon (\rho_0 + \rho_i - \rho_j -\rho_k),
\end{array}
\end{equation}
with $1\leq t\leq 3$, $\varepsilon = \pm 1$, $0\leq s\leq 3$ and $(i,j,k)\in \{ (1,2,3), ( 2,3,1), (3,1,2)\}$. For a characteristic vector $v$ which is a root vector of $\lambda_3$, 
$\mathcal H_v$ is the span of root vectors corresponding to 
{\small 
\begin{equation}\label{H7}
 \frac{1}{2}(\lambda_2 + \lambda_3) + 2\varepsilon \rho_s, \, 
\frac{1}{2}(\lambda_1 + \lambda_3) + \varepsilon (-\rho_0 + \rho_1 + \rho_2 +\rho_3), \,
\frac{1}{2}(\lambda_1 + \lambda_3) + \varepsilon (-\rho_0 + \rho_i - \rho_j -\rho_k) 
\end{equation}}
with $\varepsilon = \pm 1$, $0\leq s\leq 3$ and $(i,j,k)\in \{ (1,2,3), ( 2,3,1), (3,1,2)\}$. 
Then it is not difficult to check that $\lambda_3 - 2\phi$ is not a root (see \cite[152--154]{Drucker_1978}) for any $\phi$ which is one of \eqref{H7}.

\medskip

For $\Omega = \Omega^V_{16}$, $G = E_6$ which is an exceptional simple Lie group and its Lie algebra is $\frak g = \frak e_6$. There exists an injective Lie algebra homomorphism from $\frak e_6$ into $\frak e_7$ such that compact roots of $\frak e_7$, when restricted to the Cartan subalgebra of $\frak e_6$, are distinct roots of $\frak e_6$. The compact positive roots of $\frak e_7$ corresponding to noncompact positive roots of $\frak e_6$ are given by 
{\small 
\begin{equation}\nonumber
\frac{1}{2}(\lambda_2-\lambda_3) + 2 \varepsilon \rho_s, \,
\frac{1}{2}(\lambda_1-\lambda_3) + \varepsilon ( -\rho_1 + \rho_1 + \rho_2 + \rho_3),\,
\frac{1}{2}(\lambda_1-\lambda_3) + \varepsilon ( -\rho_1 + \rho_i - \rho_j - \rho_k),
\end{equation}} 
with $\varepsilon = \pm 1$, $0\leq s\leq 3$ and $(i,j,k)\in \{ (1,2,3), ( 2,3,1), (3,1,2)\}$. For a characteristic vector $v$ which is a root vector of $\frac{1}{2}(\lambda_2-\lambda_3) - 2\rho_0$, $\mathcal H_v$ is the span of root vectors corresponding to 
{\small 
\begin{equation}\label{H6}
\frac{1}{2}(\lambda_2-\lambda_3) + 2 \varepsilon \rho_s, \,
\frac{1}{2}(\lambda_1-\lambda_3) -\rho_1 + \rho_1 + \rho_2 + \rho_3,\,
\frac{1}{2}(\lambda_1-\lambda_3)  -\rho_1 + \rho_i - \rho_j - \rho_k,
\end{equation}}
with $1\leq s\leq 3$ and $(i,j,k)\in \{ (1,2,3), ( 2,3,1), (3,1,2)\}$. 
Hence $\frac{1}{2}(\lambda_2-\lambda_3) - 2\rho_0 - 2\phi$ is not a root (see \cite[152--154]{Drucker_1978}) for any $\phi$ which is one of \eqref{H6}.
\end{proof}

For each $x\in X$ define 
$$
\mathcal V_x := \bigcup \{\ell : \ell \text{ is a minimal rational curve on } X \text{ through } x \},
$$
and $V_x = \mathcal V_x \cap \Omega\subset \Omega$.
Let $\delta\in H^2(X,\mathbb Z)\cong \mathbb Z$ be the positive generator of the second integral cohomology group of $X$. Write 
$c_1(X) = (p+2)\delta$. 
Let $q \in\partial\Omega$ which sits on the boundary of a minimal disc.
Note that $V_q$ is the union of the minimal discs whose boundaries contain the point $q$.
In \cite{Mok_2016}, it is proved that 
$(V_q, b_\Omega|_{V_q})$ is the image of a holomorphic isometric embedding $F\colon (\mathbb B^{p+1}, b_{\mathbb B^{p+1}})
\rightarrow (\Omega,b_\Omega)$, where $\mathbb B^{p+1}:= \{ z\in \mathbb C^{p+1}: |z|<1\}$ is 
the $p+1$ dimensional unit ball.

\subsection{The Kobayashi metric on bounded symmetric domains}
We recall a few basic facts concerning the Kobayashi pseudodistance and complex/real geodesics. For more details, see \cite{Royden_1971, Kobayashi}.
Let $\Delta := \{ z\in \mathbb C : |z|<1\}$ denote the unit disc. 
Let $M$ be a complex manifold and $TM$ be its holomorphic tangent bundle. The infinitesimal Kobayashi-Royden pseudometric $k_M\colon TM\rightarrow \mathbb R$ on $M$ is a Finsler pseudometric defined by 
\begin{equation}\nonumber
k_M(z;v) = \inf \left\{ |\zeta| 
: \exists f\in \text{Hol}(\Delta, M),\, f(0)=z, \, df(0) = v/{\zeta}\right\},
\end{equation}
and for $x,\,y\in M$ the Kobayashi pseudodistance $d_M^K$ is defined by 
$$
d_M^K(x,y) = \inf \left\{ \int_0^1 k_M( \gamma(t), \gamma'(t) ) dt 
: \gamma(0) = x, \gamma(1) = y\right\}.
$$
Remark that $k_M$ and $d_M^K$ enjoy the distance decreasing property with respect to holomorphic mappings, i.e. for complex manifolds $M,\, N$ and a holomorphic map $f\colon M\rightarrow N$, we have 
$$
k_N(f(z); df(v))\leq k_M(z;v),\quad d_N^K(f(x), f(y)) \leq d_M^K(x,y)
$$
for any $(z;v)\in TM$ and $x,y\in M$.
For a complex manifold $M$, $k_M$ is upper semicontinuous. 
If $M$ is a taut manifold, i.e. $\text{Hol}(\Delta, M)$ is a normal family, then $k_M$ is continuous on $TM$. Therefore $k_\Omega$ is continuous for any bounded symmetric domain $\Omega$.
We say that $M$ is (Kobayashi) hyperbolic if and only if $k_M(z;v)>0$ whenever $v\neq 0$.
It is known that any bounded domain in $\mathbb C^n$ is hyperbolic.
\begin{example}
\begin{enumerate}
\item 
If $M$ is the unit ball $\mathbb B^n=\{ z\in \mathbb C^n : |z|<1\}$, then $k_{\mathbb B^n}$ coincides with the Bergman metric $b_{\mathbb B^n}$. 
\item
If $M$ is the polydisc $\Delta^r$, then we have 
$$k_{\Delta^r}(p;v) = \max_{1\leq j\leq r}\{b_\Delta(p_j; v_j)\}
\quad\text{ and } \quad
d_{\Delta^r}^K(x,y) = \max_{1\leq j\leq r}\{d^K_\Delta(x_j, y_j)\} $$
where $p=(p_1, \ldots, p_r)$, $v=(v_1, \ldots, v_r)$, $x=(x_1,\ldots, x_r)$ and $y=(y_1,\ldots, y_r)$.
\end{enumerate}
\end{example}

\begin{definition}
Let $M,\,N$ be Kobayashi hyperbolic manifolds.
\begin{enumerate}
\item
A map $f\colon M\rightarrow N$ is a {\it totally geodesic embedding} if $f$ is an isometry for the Kobayashi distance, i.e.
$$
d^K_M(x,y) = d^K_N(f(x), f(y))
$$
for any two points $x,y\in M$.
\item
A {\it (real) geodesic} in $M$ is a $C^1$ locally regular curve $\gamma\colon I\rightarrow M$ such that 
\begin{equation}\nonumber
 \int_{t_1}^{t_2} k_M(\gamma(t); \gamma'(t)) dt = d^K_M ( \gamma(t_1),\gamma( t_2))
\end{equation}
for all $t_1, t_2\in I$, where $I\subset \mathbb R$ is an interval.
\item 
A holomorphic map $\varphi\colon\Delta\rightarrow M$ is a {\it complex geodesic} if $\varphi$ is an isometry for the Kobayashi distances on $\Delta$ and $M$.
\end{enumerate}
\end{definition}

Suppose that $M$ is a convex domain in $\mathbb C^n$.
Every complex geodesic $\varphi\colon \Delta\rightarrow M$ gives rise to a unit speed geodesic $\gamma\colon \mathbb R\rightarrow M$ by 
$\gamma(t) = \varphi(\tanh (t))$ for any $t\in \mathbb R$.

Note that every complex geodesic is a proper injective map from $\Delta$ to $M$. 
If $M$ is a strongly convex domain in $\mathbb C^n$, then any pair of points of $M$ is contained in a unique complex geodesic and it is continuously extended to the boundary (\cite{Lempert_1981}). On the other hand if $M$ is weakly convex, we know that for any two points in $M$ there exists a complex geodesic joining them but there could be many others.
For example if $M$ is a polydisc $\Delta^k$, then a holomorphic map $\varphi = (\varphi_1, \ldots, \varphi_k)\colon \Delta\rightarrow\Delta^k$ is a complex geodesic if and only if $\varphi_j$ is an automorphism of $\Delta$ for some $j$. In particular, complex geodesic does not need to be extended continuously to the boundary.

\begin{lemma}\label{max}
Let $\Omega$ be a bounded symmetric domain. For $v\in T_p \Omega$, let $\Delta^r$ be a totally geodesic Poincar\'{e} polydisc passing through $p$ such that $v\in T_p \Delta^r$ with respect to the Bergman metric.
Then
\begin{equation}\nonumber
k_\Omega(p;v) = \max_{1\leq j\leq r} k_\Delta(p_j;v_j),
\end{equation}
where $p_j$ and $v_j$ are the $j$-th components of $p\in \Delta^r$ and $v\in T_p\Delta^r\cong \mathbb{ C}^r$, respectively.
In particular, if $[v]\in \mathcal C_p$, then
$$ k_\Omega(p;v)=b_\Omega(p;v)$$
and $\Delta^r$ is totally geodesic with respect to the Kobayashi metric.
\end{lemma}
\begin{proof}
By the distance decreasing property of Kobayashi metrics with respect to holomorphic mappings, we have 
$$\max_{1\leq j\leq r} k_\Delta(p_j;v_j)=k_{\Delta^r}(p;v)\geq k_\Omega(p;v).$$
Since each $v_j$, $j=1,\ldots, r$ has rank 1, there exists a minimal disc passing through $p$ tangential to $v_j$ and hence there exists a projection $p_j$ to the minimal disc.
This implies that $k_\Omega(p;v)\geq k_\Delta(p_j;v_j)$ for each $j$. 
\end{proof}

\section{Real and complex geodesics in bounded symmetric domains}

\begin{lemma}\label{polyD}
For a unit speed real geodesic $\gamma\colon I \rightarrow \Omega$, there exists a totally geodesic polydisc $\Delta^k$ of dimension $k$ such that 
$$ \gamma(I)\subset \Delta^k\times(\Delta^k)^\perp$$
where $ \gamma_1\colon I \rightarrow \Delta^k$ is a geodesic such that each component is a unit speed geodesic in a disc and $\gamma_2\colon I\to (\Delta^k)^\perp $ is not a geodesic. Here $ \gamma_1$ 
and 
$\gamma_2$ are the $\Delta^k$ and $(\Delta^k)^\perp$-components of $\gamma$, respectively. 
Moreover, there exists a unique totally geodesic holomorphic disc $\Delta\subset \Delta^k$ that contains $\gamma_1(I)$.
\end{lemma}
\begin{proof}
Assume $\gamma(0)=0.$ Choose the unique minimal totally geodesic polydisc $\Delta^m$ such that $\gamma'(0)\in T_{0}\Delta^m$.
Let
$$
\gamma'(0)=(v_1,\ldots,v_m)\in T_0 \Delta^m \text{ with } 1=|v_1|\geq \cdots\geq |v_m|>0.$$
First we will show that
$$\gamma(I)\subset \Delta_0\times (\Delta_0)^\perp$$
for some minimal disc $\Delta_0\subset \Delta^m$ passing through $0$.

For each $t\in I$, by Polydisc Theorem (Theorem \ref{polydisc theorem}), we can choose the minimal totally geodesic polydisc $P_t$ such that 
$$0, \gamma(t)\in P_t.$$  
Since $\gamma$ is $C^1$, we can choose a sufficiently small $t_0>0$ and a continuous family $\{\Delta_t,\,0<t<t_0\}$ of minimal discs passing through $0$ such that 
$$P_t\subset \Delta_t\times\Delta_t^\perp$$ and 
\begin{equation}\label{k-dist-2}
t=d_\Omega^K(0, \gamma(t))=d_{\Delta}^K(0, \pi_t\circ \gamma(t)),
\end{equation}
where $\pi_t:\Omega\to \Delta_t$ is the projection to $\Delta_t$. Since $\Delta_t$ is a minimal disc, we can choose a continuous family of complex geodesic $\Gamma_t:\Delta\to \Delta_t\subset \Omega$, $0<t< t_0$ such that
\begin{equation*}
\Gamma_t(0)=0,\quad  \Gamma_t(\tanh(t))=\pi_t\circ \gamma(t).
\end{equation*}

Fix $t\in (0,\,t_0)$ and
write
$$ \Gamma_t'(0)=w_t\in T_0\Delta_t.$$
Since $\Gamma_t(\cdot)$ is a complex geodesic, we obtain 
$$k_\Omega(0, \Gamma_t'(0))=1,$$
i.e. $|w_t|=1.$ 
Choose a maximal rank one totally geodesic subdomain $L\subset\Omega$ through $0$ such that 
$w_t \in T_0L.$ 
Let
$$
v_L:=(\pi_L)_*\gamma'(0)=(v_0, v_H)\in \mathbb Cw_t\oplus {H}_{[w_t]},
$$
where $\pi_L$ is the  projection to $L$.
Since $ v_L\in \mathcal C_0(\Omega)$, we obtain
$$ k_\Omega(0,  v_L)=b_\Omega(0,  v_L)=\sqrt{|v_0|^2+|v_H|^2}.$$
On the other hand by \eqref{k-dist-2}, we obtain
$$
k_\Omega(0,\gamma'(0))=k_\Delta(0, (\pi_{t}\circ\gamma)'(0))=|v_0|.
$$
By the distance decreasing property of Kobayashi metrics, we obtain
$$ |v_0|=k_\Omega(0, \gamma'(0)) \geq k_L(0, (\pi_L)_*\gamma'(0))=\sqrt{|v_0|^2+|v_H|^2}.$$
Therefore we obtain
$$v_H=0,\quad |v_0|= k_\Omega(0, v_0)=1.$$
Moreover, by Lemma \ref{ball} without loss of generality we may choose a totally geodesic rank one subspace $L$, such that $\pi_L\circ \gamma$ is the unique geodesic in $L$ with $\pi_L\circ\gamma(0)=0$, $(\pi_L\circ \gamma)'(0)=v_0.$ 
Since $v_H=0$, this implies 
$$\pi_L\circ \gamma([0,\,t])\subset \Delta_{t}.$$
Since $L$ is arbitrary, by Lemma \ref{ball} we obtain 
$$ \gamma([0,\,t])\in \Delta_{t}\oplus (\Delta_{t})^\perp,$$
implying that
$$ \gamma([0,\,t_0])\in \Delta_{t_0}\oplus (\Delta_{t_0})^\perp$$
and 
$(\Delta_{t_0})^\perp$-component $\gamma_0^\perp$ of $\gamma$ is a real geodesic in $\Delta_{t_0}^\perp$
if and only if 
$$k_\Omega(0,(\gamma_0^\perp)'(s))=1$$
for any $s\in [0,\,t_0]$.
Hence by induction argument, 
we obtain
$$\gamma([0,\,t_0])\subset  \Delta_0^k\times (\Delta_0^k)^\perp$$
for some $\Delta_0^k\subset \Delta^m$.
By iterating this process, we obtain
$$\gamma(I)\subset \Delta_0^k\times (\Delta_0^k)^\perp$$
for another possibly smaller dimensional polydisc $\Delta^k_0$.
The rest part of the lemma is clear by the argument above.
\end{proof}

\begin{corollary}\label{unique minimal disc}
Let $\Omega$ be a bounded symmetric domain and $\gamma\colon I\rightarrow \Omega$ be a real geodesic such that 
$\gamma(I)$ is an integral submanifold of $\mathcal{RC}(\Omega)$.
Then $\gamma(I)$ is contained in a unique minimal disc.
\end{corollary}

\begin{proof}
We may assume that $\gamma(0)=0\in \Omega$ and $|\gamma'(0)|=1$.
Since $\gamma'(0)\in \mathcal {RC}_0$, there exists a minimal disc $\Delta_0$ such that $\gamma'(0)\in T_0\Delta_0$. 
Then by Lemma \ref{polyD}, we obtain
\begin{equation}\nonumber
 \gamma(I)\subset \Delta_0\times\Delta_0^\perp
 \end{equation}
 and 
 \begin{equation}\nonumber
 k_\Omega(\gamma(t), \gamma'(t))=k_\Delta(\gamma_0(t), \gamma_0'(t))
\end{equation}
where we let
$\gamma=(\gamma_0, \gamma_0^\perp)\in \Delta_0\times \Delta_0^\perp$.
Since $\gamma'(t) \in \mathcal {RC}_{\gamma(t)}$ for any $t\in I$, we obtain
$$(\gamma^\perp_0)'\equiv 0,$$
which completes the proof.
\end{proof}

\begin{proposition}\label{poly-K}
Let $\Delta^m=\Delta_1\times\cdots\times\Delta_m$ be an $m$-dimensional polydisc and let 
$f:\Delta^m\to\Delta^m$ be a $C^1$ isometry with respect to the Kobayashi metric.
Then up to automorphisms of $\Delta^m$, $f$ is of the form
$$ f(\zeta_1,\ldots,\zeta_m)=(\Phi_1(\zeta_1),\ldots,\Phi_m(\zeta_m))$$
with each $\Phi_i$ being $\zeta_i\mapsto \zeta_i$ or $\zeta_i\mapsto \bar\zeta_i$. 
\end{proposition}

\begin{proof}
We will use induction on $m$. If $m=1$, then it is already known. Now assume that the proposition holds for $m-1\geq 1$. 

Let $f=(f_1,\ldots, f_m)$ and $\gamma:\mathbb R\to \Delta_1$ be a unit speed complete geodesic. Then for any $y\in \Delta_1^\perp=\Delta_2\times\cdots\times\Delta_m$, the map $t\to (\gamma(t),y)$ is a complete geodesic in $\Delta_1\times\cdots\times\Delta_m$. Hence by Lemma~\ref{polyD}, there exists $j$ depending on $y$ such that $f_j(\gamma(\cdot), y):\mathbb R\to \Delta_j$ is a complete geodesic. Since $f$ is $C^1$, by choosing general $\gamma$ and $y$, we may assume that
$f_j(\gamma(\cdot), y)$ is a complete geodesic for all $y$ in an open set $U\subset \Delta_1^\perp$.   
Moreover, after composing automorphisms of $\Delta^m$, we may assume that $\gamma(0)=0$, $0\in U$, $f_1(\gamma(t), 0)=\gamma(t)$ and $f_1(e^{i\theta}\gamma(\cdot), y)$ is a unit speed complete geodesic in $\Delta_1$ for all $(y, \theta)\in U\times I$ for some open interval $I\subset (-\pi,~\pi]$ containing $0$. 

 Since $f_1(0,0)=0$ by assumption, if $\theta\in I$, then there exists $\eta$ such that
$$f_1(e^{i\theta}\gamma(t),0)=e^{i\eta}\gamma(t)$$
for all $t$.
Since $f_1(\gamma(t),0)=\gamma(t)$, we have 
\begin{equation}
    \begin{aligned}
    d_\Delta(\gamma(t), e^{i\theta} \gamma(t)) &= d_{\Delta^m} ((\gamma(t), 0), (e^{i\theta} \gamma(t), 0)) 
    =d_{\Delta^m} (f(\gamma(t), 0), f(e^{i\theta} \gamma(t), 0))\\
    &\geq d_{\Delta} (f_1(\gamma(t), 0), f_1(e^{i\theta} \gamma(t), 0))
    = d_\Delta(\gamma(t), e^{i\eta}\gamma(t))
    \end{aligned}
\end{equation}
and by a similar way
$$d_\Delta(\gamma(t), e^{i\theta}\gamma(-t))\geq d_\Delta(\gamma(t), e^{i\eta}\gamma(-t)).$$
Therefore $\eta=\pm \theta$.
Assume that $f_1(\cdot, 0)$ is orientation preserving near $0$. Then $\eta=\theta$ for all $\theta\in I$. 
Define a sector
$$C:=\bigcup_{\theta\in I}e^{i\theta}\gamma(\mathbb R).$$
Then we obtain
$$f_1(\zeta, 0)=\zeta$$
for all $\zeta\in C$.
Similarly, we can show that for each $y\in U$, there exists an automorphism $\phi_y$ of $\Delta$ such that
$$
f_1(\zeta, y)=\phi_y(\zeta)
$$
for all $\zeta\in C$.

On the other hand, since $f$ is totally geodesic, we obtain
$$ 
d^K_{\Delta^{m}}(f(\zeta,x), f(\zeta,y))
=d^K_{\Delta^{m}}((\zeta,x),(\zeta, y))
=d^K_{\Delta^{m-1}}(x,y)
<\infty$$
for any  $x, y\in \Delta_1^\perp$.
Therefore the limit set $\lim_{t\to \infty}f(e^{i\theta}\gamma(t), \Delta_1^\perp)$ is contained in the closure of a unique boundary component $C_\theta\subset\partial\Delta_1\times{\Delta_1^\perp}$ depending only on $\theta$.  
It implies that for all $y\in U$,
\begin{equation}\nonumber
\lim_{t\to \infty}\phi_y(e^{i\theta}\gamma(t))=\lim_{t\to\infty}f_1(e^{i\theta}\gamma(t), y)=\lim_{t\to\infty}f_1(e^{i\theta}\gamma(t), 0)
\end{equation}
for all $\theta\in I$
and therefore 
$$\phi_y(\zeta)=\zeta.$$
That is, on $C\times U$, $f_1(\zeta, y)$ is independent of $y$. 

Let $\widetilde f:=(f_2,\ldots,f_m):\Delta^m\to \Delta_2\times\cdots\times\Delta_m$. Since 
$$\partial_yf_1(\zeta, y)=0,\quad (\zeta, y)\in \overline{C\times U},$$
there exists an open set $W\subset \Delta_1$ containing $C$ such that 
for any complete geodesic $\widetilde \gamma$ in $\Delta_1^\perp$ passing through a point in $\overline U$, $\widetilde f(\zeta, \widetilde\gamma(\cdot))$, $\zeta\in W$ is a complete geodesic in $\Delta_1^\perp$ and the limit set $\lim_{t\to\infty}f(\Delta_1,\widetilde\gamma(t))$ is contained in the closure of a boundary component in $\Delta_1\times \partial\Delta_1^\perp$.
Then by the same argument, we can choose an open cone $V\subset\Delta_1^\perp$ with vertex $0\in U$ such that on $V$, 
\begin{equation}\label{W}
 \widetilde f(\zeta, \cdot)= \widetilde f(0,\cdot)
 \end{equation}
for all $\zeta\in W$.
Hence if $y\in V\cap U$, then for any geodesic $\Gamma:\mathbb R\to W,$ 
$$\widetilde f(\Gamma(t),y)=\widetilde f(0, y).$$

Let $x\in \Delta_1^\perp$ and $y\in V\cap U$. Since for any geodesic $\Gamma:\mathbb R\to W,$ 
we have
\begin{equation}\nonumber
\begin{aligned}
&d_{\Delta_1^\perp}(\widetilde f(\Gamma(t), x), \widetilde f(0,y))=
d_{\Delta_1^\perp}(\widetilde f(\Gamma(t), x), \widetilde f(\Gamma(t),y))\\
&\quad\quad\quad\leq d_{\Delta^m}(f(\Gamma(t), x), f(\Gamma(t),y))
= d_{\Delta^m}((\Gamma(t), x), (\Gamma(t),y))<\infty,
\end{aligned}
\end{equation}
for any $x\in \Delta_1^\perp$, $ \widetilde f(\Gamma(\cdot), x)$ is not a geodesic and therefore $f_1(\Gamma(\cdot), x)$ should be a complete geodesic in $\Delta_1$. Consider a family of complete geodesic $\Gamma_\theta:=e^{i\theta}\gamma:\mathbb R\to C$, $\theta\in  I.$ Then by the same argument, we obtain that on $C\times\Delta_1^\perp$, $f_1(\zeta, x)$ is independent of $x$, i.e.
$$f_1(\zeta, x)=f_1(\zeta, 0)=\zeta,\quad(\zeta,x)\in C\times\Delta_1^\perp$$
and $\widetilde f(\zeta,\cdot):\Delta_1^\perp\to\Delta_1^\perp$ is an isometry for all $\zeta\in C$. By induction argument, there exist automorphisms $\Phi_\zeta, \Psi_\zeta$ of $\Delta_1^\perp$ such that $\Psi_\zeta\circ\widetilde f(\zeta,\Phi_\zeta(\cdot))$ is a desired form. On the other hand by \eqref{W}, $\widetilde f(\zeta,\cdot)$ is equal to $\widetilde f(0,\cdot)$ on an open cone $V$.
Therefore 
we obtain
\begin{equation}\label{W2}
\widetilde f(\zeta,x)=\widetilde f(0,x)
\end{equation}
for all $(\zeta,x)\in C\times \Delta_1^\perp$.

Choose a complete geodesic $\Gamma:\mathbb R\to \Delta_1$ that passes a point in the interior of $C$. By \eqref{W2},  $\widetilde f(\Gamma(\cdot),x)$ is not a geodesic for all $x\in \Delta_1^\perp$. Therefore $f_1(\Gamma(\cdot), x)$ is a complete geodesic in $\Delta_1$ such that 
if $\Gamma(t)\in C$, then
$$f_1(\Gamma(t),x) =f_1(\Gamma(t), 0)=\Gamma(t),$$ 
implying that
$$f_1(\zeta,x)=\zeta$$
for all $(\zeta,x)\in \Gamma(\mathbb R)\times \Delta_1^\perp$.
Since $\Gamma$ is arbitrary, we obtain
$$f_1(\zeta,x)=\zeta$$
for all $(\zeta,x)\in \Delta_1\times \Delta_1^\perp$
and $\widetilde f(\zeta,\cdot):\Delta_1^\perp\to\Delta_1^\perp$ is an isometry for all $\zeta\in \Delta_1$.

Let $\gamma$ be a complete geodesic in $\Delta^\perp_1$. Then 
$$f(\zeta, \gamma(t))=(f_1(\zeta,\gamma(t)), \widetilde f(\zeta, \gamma(t)))=(\zeta, \widetilde f(\zeta,\gamma(t))).$$
Therefore the limit set $\lim_{t\to\infty}f(\Delta_1, \gamma(t))$ should be contained in a boundary component in $\Delta_1\times\partial\Delta_1^\perp$ depending only on $\gamma$. Since $\widetilde f(\zeta,\cdot)$ is a desired form by induction, this implies that $\widetilde f$ is independent of $\zeta$, i.e.
$$\widetilde f(\zeta,\cdot)=\widetilde f(0,\cdot),$$
which completes the proof.
\end{proof}

\begin{corollary}\label{poly-KK}
Let $\Delta^m$ be an $m$-dimensional polydisc and let 
$f:\Delta^m\to\Delta^m$ be a $C^1$ isometry with respect to the Kobayashi metric.
Then $f$ extends continuously to $\partial\Delta^m$ and if two $C^1$ isometry $f, g:\Delta^m\to \Delta^m$ coincide on an open set in the Shilov boundary of $\Delta^m$, then $f\equiv g$.
\end{corollary}

\section{Proof of theorems}

\begin{lemma}\label{holo-disc}
Let $\Omega$ be an irreducible bounded symmetric domain and let $F\colon\Omega\to\mathbb C^n$ be a $C^1$ map such that $F_*(v)\neq 0$ for all $v\neq 0.$
If $F$ is holomorphic or anti-holomorphic on every minimal disc, then $F$ is either holomorphic or anti-holomorphic.
\end{lemma}
\begin{proof}
Let $F\colon\Omega\to \mathbb C^n$ be a $C^1$-map. 
Define
$$U:=\{[v]\in \mathcal C(\Omega): \bar v F=0\}.$$
By assumption, we may assume that $U$ is a nonempty set. Then it is enough to show that $U=\mathcal C(\Omega)$. By continuity of $\bar \partial F$, $U$ is a closed set. Assume that $\partial U\neq\emptyset.$ Choose $[v]\in \partial U$. Then by continuity of $dF$, we obtain 
$$v F=\bar v F=0,$$
i.e. 
$$dF( \text{Re }v)=dF(\text{Im }v)=0.$$
Since $dF\neq 0$, this does not happen. Since $\Omega$ is irreducible, we obtain $U=\mathcal C(\Omega)$. 
\end{proof}
 
\begin{lemma}\label{component_totally}
Let $\Omega$ and $\Omega'$ be bounded symmetric domains and let $f\colon\Omega\to \Omega'$ be a $C^1$-smooth totally geodesic isometric embedding with respect to their Kobayashi metrics. Suppose 
$${\rm rank }(\Omega)\geq {\rm rank }(\Omega').$$
Then 
$$ {\rm rank }(\Omega)={\rm rank }(\Omega')$$
and for any totally geodesic polydisc $\Delta^r$ in $\Omega$,
$f(\Delta^r)$ is contained in a maximal totally geodesic polydisc in $\Omega'$. 
\end{lemma}

\begin{proof}
Let $r$ be the rank of $\Omega$. We will use induction on $r$. If $r=1$, then it is trivial. Let $r>1$.
It is enough to show it for a general maximal totally geodesic polydisc $\Delta^r\subset \Omega$ passing through $0$. 
Let $v\in T_0\Delta^r$ be a general vector realizing the maximal holomorphic sectional curvature. 
After rotation, we may assume
\begin{equation}\label{v}
v=(1,\ldots,1)\in T_0(\Delta_1\times \cdots\times \Delta_r).
\end{equation}
For each $i=1,\ldots,r$ and a point $(0,x)\in  \{0\}\times \Delta_{i}^\perp\subset \Omega$,
let $\gamma_i(\cdot,x)=\gamma_i^v(\cdot,x)\colon\mathbb R\to \Delta_i\times\Delta_i^\perp$ be a unit speed complete real geodesic with respect to the Bergman metric such that 
$$\gamma_i(0,x)=(0,x),\quad \gamma_i'(0, x)=(1,0,\ldots, 0)\in T_{(0,x)}(\Delta_i\times x).$$
Since $f$ is totally geodesic, 
by Lemma \ref{polyD}, there exists a unique totally geodesic polydisc $Q_i(x)=Q_i^v(x)\subset \Omega'$ passing through $f(0, x)$ such that
\begin{equation}\label{del-x}
f\circ \gamma_i(t, x)=(\Phi_i(t, x),\Psi_i(t, x))\in Q_i(x)\times Q_i(x)^\perp,\quad t\in \mathbb R,
\end{equation} 
where each component of $\Phi_i(\cdot, x)$ is a geodesic of the disc and $\Psi_i(\cdot,x)$ is not a geodesic. Since $f$ is $C^1$, by continuity of the derivatives of $f$, we may assume that the dimension of $Q_i(x)=Q_i^v(x)$ is constant on some open neighborhood of $v\in T_0\Omega$.

Now assume $i=1$.
Since $f$ is an isometry, for each $x\in \Delta_1^\perp$, 
$$ d_{\Omega'}^K(f\circ \gamma_1(t, 0), f\circ\gamma_1(t,x))=d_\Omega^K(\gamma_1(t, 0),\gamma_1( t, x))=d_{\Delta_1^\perp}^K(0, x)<\infty.$$
Hence 
the limit set $\lim_{t\to \infty}f\circ \gamma_1(t, \Delta_1^\perp) $ should be contained in a unique boundary component $\Sigma$ of $\Omega'$. Choose a maximal totally geodesic polydisc $R(x)\subset \Omega'$ passing through $f(0, x)$ such that $R(x)\times \Sigma\subset \Omega'$ is totally geodesic. Since $\Sigma$ is independent of $x$, $R(x)$ is parallel with $R(0)$. Moreover, since 
$$T\Sigma\subset  \mathcal N_{[\partial_t\Phi_1(t, x)]},$$   
$Q_1(x)$ is contained in $R(x)$. Since $R(x)$ is parallel with $R(0)$, by continuity of the derivatives of $f$, we may assume that $Q_1(x)$ is parallel with $Q_1(0)$ for all $x\in \Delta_1^\perp$ sufficiently close to $0$. 
In particular, 
\begin{equation}\label{del'}
f\circ \gamma_1(\mathbb R, U_1)\subset Q_1\times Q_1^\perp
\end{equation}
on an open set $U_1\subset \Delta_1^\perp$ containing $0$, where $Q_1=Q_1(0).$ Since $\Sigma$ and $R(0)$ is independent of $x$, there exists possibly smaller dimensional polydisc $Q_1\subset R(0)$ such that
$$ f\circ \gamma_1(\mathbb R, \Delta_1^\perp)\subset Q_1\times Q_1^\perp.$$

We may assume that $Q_1$ is the maximal polydisc such that each component of $\Phi_1(\cdot,x)\colon\mathbb R\to Q_1$ is a unit speed complete geodesic in a disc for all $x\in \Delta_1^\perp$. Note that since 
$$f\circ \gamma_1=(\Phi_1,\Psi_1):\mathbb R\times\Delta_1^\perp\to Q_1\times Q_1^\perp$$ 
is totally geodesic isometric embedding, $\Phi_1$ and $\Psi_1$ are distance decreasing maps.
Therefore for any $x,y\in \Delta_1^\perp$,
\begin{equation}\label{inequal}
\begin{aligned}
d^K_{\Delta_1^\perp}(x,y) 
&= d_\Omega^K(\gamma_1(t,x), \gamma_1(t,y))
=d^K_{\Omega'}( f\circ\gamma_1(t,x), f\circ\gamma_1(t,y))\\ 
&= d^K_{Q_1\times Q_1^\perp} ((\Phi_1(t,x), \Psi_1(t,x)), (\Phi_1(t,y), \Psi_1(t,y))) \\
&\geq  d^K_{Q_1}(\Phi_1(t,
x), \Phi_1(t,y)).
\end{aligned}
\end{equation}
Since each component of  $\Phi_1(\cdot, x)\colon \mathbb R\to Q_1$ is a unit speed complete geodesic in a disc for all $x\in \Delta_1^\perp$, \eqref{inequal} implies
$$
\lim_{t\to\infty} \Phi_1(t,x) = \lim_{t\to\infty} \Phi_1(t,y) \quad \text{ and } \lim_{t\to -\infty} \Phi_1(t,x) = \lim_{t\to -\infty} \Phi_1(t,y)$$
and hence
\begin{equation}
\Phi_1(\mathbb R,x) = \Phi_1(\mathbb R,y),
\end{equation}
i.e.
\begin{equation}\label{open}
f\circ\gamma_1(\mathbb R, \Delta^\perp_1)\subset \Phi_1(\mathbb R,0)\times Q_1^\perp.
\end{equation}

We will show that 
$$\Psi_1(t,\cdot):\Delta_1^\perp\to Q_1^\perp$$
is a totally geodesic isometric embedding for all $t\in \mathbb R$. 
Assume otherwise. Then by continuity of the derivatives of $f$, there exists an open interval $I\subset \mathbb R$ such that $\Psi_1(t, \cdot)$ is not a totally geodesic map for all $t\in I\subset \mathbb R$. After shrinking $I$ if necessary, we may assume that there exists  a $C^1$ complete geodesic $\gamma_0:\mathbb R\to \Delta_1^\perp$, an open set $U\subset \mathbb R$ and small $\epsilon,\eta>0$ such that  
$$d^K_{Q_1^\perp}(\Psi_1(t,\widetilde\gamma(s_1)), \Psi_1(t, \widetilde\gamma(s_2))\leq (1-\eta)|s_1-s_2|,
\quad t\in I,~s_1, s_2\in U$$ 
for all geodesic $\widetilde \gamma$ such that $\|\gamma_0-\widetilde\gamma\|_{C^1(U)}<\epsilon.$ 
Therefore $\Phi_1(t, \widetilde\gamma)$ should be a complete geodesic for all such geodesics $\widetilde \gamma.$
This implies that
there exists an open set $\widetilde U\subset \Delta_1^\perp$ such that $\Phi_1(t,\cdot)$ is locally one to one on $\widetilde U$. On the other hand, by \eqref{open}, we obtain 
$$\Phi_1(\mathbb R\times \Delta_1^\perp)=\Phi_1(\mathbb R,0),$$
which is a contradiction. 
We have seen that $\Psi_1(t, \cdot)\colon\Delta_1^\perp\to Q_1^\perp$ is a totally geodesic isometric embedding for all $t\in \mathbb R$. Hence by induction argument, $Q_1^\perp$ is of rank $(r-1)$. Since 
$${\rm rank }(\Omega')=\dim Q_1+{\rm rank }(Q_1^\perp)\leq {\rm rank }(\Omega)=r,$$ 
$Q_1$ is a minimal disc and 
$${\rm rank }(\Omega')={\rm rank }(\Omega).$$

Let $P_1:=\Delta_2\times\cdots\times\Delta_r$. Since $\Psi_1(t, \cdot):\Delta_1^\perp\to Q_1^\perp$ is a totally geodesic isometric embedding, by induction argument, there exists a $(r-1)$-dimensional totally geodesic polydisc $\widetilde P_1(t)\subset Q_1^\perp$ such that $\Psi_1(t, P_1)\subset \widetilde P_1(t)$. On the other hand, since $f$ is an isometry, for any unit speed complete geodesic $\gamma$ in $P_1$, the limit set $\lim_{t\to\infty}f\circ\gamma_1(\mathbb R, \gamma(s))$ should be contained in a unique boundary component, which implies that
$$\lim_{s\to\infty}\Psi_1(a, \gamma(s))=\lim_{s\to\infty}\Psi_1(b, \gamma(s)) ,\quad a, b\in \mathbb R.$$
Since $\gamma$ is arbitrary, on $P_1$, we obtain
\begin{equation}\label{Psi-id}
\Psi_1(a, \cdot)=\Psi_1(b, \cdot)
\end{equation}
by Corollary \ref{poly-KK} with induction argument and 
$$f\circ \gamma_1(\mathbb R,P_1)\subset Q_1\times\widetilde P_1,$$
where $\widetilde P_1:=\widetilde P_1(0).$

Choose another vector 
$$v_\theta=(e^{i\theta},1,\ldots,1)\in T_0 (\Delta_1\times\cdots\times\Delta_r).$$
Since 
$$\gamma_i^{v_\theta}=\gamma_i^v,\quad i=2,\ldots,r,$$
by the same argument as above, we obtain
$$ f\circ\gamma_1^{v_\theta}(\mathbb R\times P_1)\subset Q_1(\theta)\times \widetilde P_1(\theta)$$ 
for some minimal disc $Q_1(\theta)$ and $(r-1)$-dimensional totally geodesic polydisc $\widetilde P_1(\theta)$.
Since 
$$\{0\}\times\widetilde P_1=f\circ \gamma_1(0,P_1)=f\circ\gamma_1^{v_\theta}(0,P_1)=\{0\}\times\widetilde P_1(\theta),$$
we obtain 
$$\widetilde P_1=\widetilde P_1(\theta)$$
for any $\theta$.
Let $B_1$ be the maximal totally geodesic subdomain such that
$$T_0B_1=\mathcal N_{[T_0\widetilde P_1]}.$$
Then
$$T_0Q_1(\theta)\subset T_0B_1.$$
Since $\theta$ is arbitrary, we obtain
$$ f(\Delta_1\times P_1)\subset B_1\times \widetilde P_1.$$

Now consider $f\circ \gamma_2$. Then by the same argument, we obtain
$$f(\Delta_2\times P_2)\subset B_2\times \widetilde P_2,$$
where $P_2=\Delta_1\times\Delta_3\times\cdots\times\Delta_r$
and $\widetilde P_2$ component of $f$ is an isometry on $\{\zeta_2\}\times P_2$ and independent of $\zeta_2\in \Delta_2$.
Since
$$\{0\}\times \widetilde P_1\subset B_2\times \widetilde P_2,$$
there exists a minimal disc $\widetilde \Delta\subset B_2$ such that
$$\{0\}\times \widetilde P_1\subset \widetilde \Delta\times \widetilde P_2.$$
Suppose 
$$T_0\widetilde \Delta\cap T_0\widetilde P_1= \{0\}.$$
Then $\widetilde P_1=\widetilde P_2$.
On the other hand, by \eqref{Psi-id}, $\Psi_1$ is independent of $\zeta_1\in \Delta_1$, contradicting the assumption that $\widetilde P_2$ component of $f$ is an isometry on $\{\zeta_2\}\times P_2.$
Therefore we obtain 
$$ T_0\widetilde P_1\cap T_0 B_2\neq \{0\}.$$
Since $\Delta_1\times P_1=\Delta_2\times P_2=\Delta_1\times\cdots\times\Delta_r$, we obtain
$$ f(\Delta_1\times\cdots\times\Delta_r)\subset \left(B_1\times \widetilde P_1\right)\cap \left(B_2\times \widetilde P_2\right).$$
Therefore one has 
$$ f(\Delta_1\times\cdots\times\Delta_r)\subset \left(B_2\cap \widetilde P_1\right)\times \widetilde P_2,$$
which completes the proof.
\end{proof}

\begin{proof}[Proof of Theorem \ref{main}]
Let $r$ be the rank of $\Omega$. We will use induction on $r$. If $r=1$, it is well-known (cf.\cite{Antonakoudis_2017}). Assume that $r>1$. By Proposition~\ref{poly-K}, Lemma~\ref{holo-disc} and Polydisc Theorem (Theorem \ref{polydisc theorem}), it is enough to show that for a maximal totally geodesic polydisc $\Delta^r$ in $\Omega$ containing $0$,
$f(\Delta^r)$ is contained in a maximal totally geodesic polydisc in $\Omega'$. Hence by Lemma~\ref{component_totally}, we can complete the proof.
\end{proof}

\begin{proof}[Proof of Theorem \ref{Kobayashi isometry}]
We will use induction on $m$. If $m=1$, then it is true by Theorem \ref{main}. Suppose $m>1$. Assume that $\Omega_1$ is an irreducible factor whose rank is maximum among $\Omega_i$'s. We may further assume that the dimension of $\Omega_1$ is maximum among all irreducible factors with the maximum rank. Since $F$ is an isometry, by Proposition~\ref{poly-K}, Lemma~\ref{holo-disc} and Lemma~\ref{component_totally}, for each $x\in \widetilde\Omega:=\Omega_2\times\cdots\times\Omega_m$, $F(\cdot, x):\Omega_1\to \Omega_1\times\cdots\times\Omega_m$ is either holomorphic or anti-holomorphic and preserves rank one vectors. Let $F=(f_1,\ldots,f_m)$. Since $\mathcal C_x(\Omega)$ is a disjoint union of $\mathcal C_{x}(\Omega_i)$ where we regard $\Omega_i$ as a canonically embedded submanifold in $\Omega$, by continuity of the derivatives of $F$, there exists unique $i$ such that if $j\neq i$, then
$$v f_j(\cdot, x)= 0,\quad  [v]\in \mathcal C(\Omega_1).$$
Therefore
$f_i(\cdot, x)\colon\Omega_1\to \Omega_i$ is a totally geodesic isometric embedding. We may assume that $i$ is independent of $x$. 
Then by Theorem~\ref{main}, the rank of $\Omega_i$ is equal to the rank of $\Omega_1$ and $f_i(\cdot, x)\colon\Omega_1\to \Omega_i$ is either holomorphic or anti-holomorphic proper embedding. Since the dimension of $\Omega_1$ is maximum among all irreducible factors with maximum rank, $\Omega_i$ is biholomorphic to $\Omega_1$. Hence up to permutation, we may assume that $i=1$ and $f_1(\cdot, x)\in {\rm Aut}(\Omega_1)$ for all $x\in \widetilde\Omega$.  

Let $x\in \widetilde\Omega$ be fixed. Since $F$ is isometric, we obtain
$$ 
d^K_{\Omega}(F(\zeta, x), F(\zeta, 0))
=d^K_{\Omega}((\zeta, x), (\zeta,0))
\leq d^K_{\widetilde \Omega}(x,0) 
< \infty
$$
for all $\zeta\in \Omega_1$.
Since $F$ is $C^1$, by continuity of the derivatives of $F$, we may assume that $F(\Omega_1\times\{x\})$ is parallel with $\Omega_1\times \{0\}$ for some open neighborhood of $0$ and hence we have
$$\lim_{\zeta\to p}f_1(\zeta, x)=\lim_{\zeta\to p}f_1(\zeta,0)$$
for all $p\in \partial\Omega_1$.
Since $f_1(\cdot,x)\in {\rm Aut}(\Omega_1)$, this implies that
$$f_1(\cdot,x)=f_1(\cdot, 0)$$
for all $x\in \widetilde\Omega$,
i.e. $f_1$ is independent of $x\in \widetilde\Omega$. Therefore $\widetilde f(\zeta,\cdot):=(f_2(\zeta,\cdot),\ldots,f_m(\zeta,\cdot)):\widetilde\Omega\to\widetilde\Omega$ is a totally geodesic isometry for all $\zeta\in \Omega_1$. Hence by induction argument and the argument similar to the above, we obtain
$$\widetilde f(\zeta, \cdot)=\widetilde f(0,\cdot)$$
for any $\zeta\in \Omega_1$,
which completes the proof.
\end{proof}

\begin{proposition}
\label{1-disc}
Let $\Omega$ be a bounded symmetric domain and let $f:\Delta\to \Omega$ be a $C^1$-smooth totally geodesic isometric embedding that extends continuously to the boundary. If 
$f$ is tangential to $\mathcal {RC}(\Omega)$,
then 
$f$ is either holomorphic or anti-holomorphic.
\end{proposition}
\begin{proof}
Choose $z\in \partial \Delta$ and consider the geodesics, say $\gamma$, whose one of end points are $z$. 
Note that $\Delta$ is foliated by the images of such geodesics.
By Corollary \ref{unique minimal disc}, $f\circ\gamma(I)$ is contained in a unique minimal disc of $\Omega$ and hence $f(\Delta)$ is contained in $V_{f(z)}$ which is the image of a holomorphic isometric embedding $F\colon \mathbb B^{p+1}\rightarrow \Omega$ (see  \cite{Mok_2016}). For $(x,v)\in T_x\Delta$, we have 
\begin{equation}\nonumber
\begin{aligned}
k_\Delta(x;v) 
= k_\Omega(f(x); df_x(v))
&\leq b_\Omega(f(x); df_x(v)) \\
&= b_{\mathbb B^{p+1}}(F^{-1}\circ f(x), d(F^{-1}\circ f)_x(v))\\
&= k_{\mathbb B^{p+1}} (F^{-1}\circ f(x), d(F^{-1}\circ f)_x(v)).
\end{aligned}
\end{equation}
Let $\Delta_x\subset \Omega$ be the minimal disc passing through $f(x)$ tangential to $df_x(v)$. By the distance decreasing property, 
\begin{equation}\nonumber
\begin{aligned}
k_\Delta(x;v) 
&= k_\Omega(f(x); df_x(v))
\geq d_{\Delta_x}(f(x);df_x(v))\\
&\geq k_{F(\mathbb B^{p+1})}(f(x); df_x(v))\geq k_{\mathbb B^{p+1}} (F^{-1}\circ f(x), d(F^{-1}\circ f)_x(v)),
\end{aligned}
\end{equation}
and as a result $$
k_\Delta(x;v) 
=k_{\mathbb B^{p+1}} (F^{-1}\circ f(x), d(F^{-1}\circ f)_x(v)).
$$
By \cite{Antonakoudis_2017, Gaussier_Seshadri_2013}, $F^{-1}\circ f$ is either holomorphic or anti-holomorphic and hence $f$ has the same property as well.
\end{proof}

\begin{corollary}\label{ball to BSD}
Let $f:\mathbb B^n\to \Omega$ be a $C^1$ totally geodesic isometric embedding extending continuously to the boundary such that 
$$[f_*(v)]\in \mathcal{RC}(\Omega),\quad  $$
for any $v\neq 0\in T\mathbb B^n$.
Then $f$ is either holomorphic or anti-holomorphic. Moreover, $f$ is a standard embedding.
\end{corollary}

\begin{proof}[Proof of Theorem \ref{main-2}]
Let $r$ be the rank of $\Omega$. We will use induction on $r$. For $r=1$, it is proved in Corollary \ref{ball to BSD}. Assume that $r>1$. Let $p\in \Omega$ be a general point. We may assume $p=0$ and $f(0)=0$. We will show that $f$ maps $r$-dimensional totally geodesic polydisc to $r$-dimensional totally geodesic polydisc and
therefore 
$$ [f_*(v)]\in \mathcal {C}_0(\Omega')$$
for all $[v]\in \mathcal {C}_0(\Omega)$.

For a given rank $(r-k)$, $k<r$ boundary component $C$ of $\Omega$, choose a totally geodesic holomorphic disc $\Delta_C\subset \Omega$ passing through $0$ such that  
$$T C=\mathcal N_{[T_0\Delta_C]}.$$
Let $\gamma:\mathbb R\to \Delta_C$ be a complete geodesic passing through $0\in \Delta_C$ such that 
$$p_\gamma:=\lim_{t\to \infty}\gamma(t)\in C.$$
Since $f\circ \gamma$ is a complete geodesic, there exists a unique boundary component $D$ of $\Omega'$ such that 
$f(p_\gamma)$ is contained in $D$. Moreover, since $f$ is an isometry, $f(C)$ should be contained in the same boundary component $C$. Hence $D$ depends only on $C$. Let $C_1$ be another boundary component of $\Omega$ of rank $(r-k)$ such that $\overline{C}\cap\overline{C_1}$ is a rank $(r-k-1)$ boundary component. By continuity of $f$, we obtain $f(\overline{C}\cap \overline{C_1})\subset \overline{D}\cap\overline{D_1}, $ 
where $D_1$ is the boundary component of $\Omega'$ that contains $f(C_1)$. 
Furthermore, since $C\neq C_1$, the distance between any two points $p\in C, q\in C_1$ and hence the distance of any $f(p)$, $p\in C$ and $f(q)$, $q\in C_1$ is infinite. Therefore $D\neq D_1$ unless $D$ and $D_1$ are points in the Shilov boundary of $\Omega'$. In the latter case, $f$ is constant on $C$ and $C_1$. Since $\overline D\cap\overline{D_1}\neq \emptyset$, we obtain $D=D_1$, i.e. $f(C)=f(C_1).$

Suppose that there exists an open set $U$ in the rank $(r-k)$ boundary orbit of $\Omega$ such that if a rank $(r-k)$ boundary component $C$ of $\Omega$ satisfies $C\cap U\neq\emptyset$, then $f$ is constant on $C$. Choose two rank $(r-k)$ boundary components $C_1, C_2$ such that $C_1\cap U$ and $C_2\cap U$ are not empty and $\overline C_1\cap \overline C_2\neq\emptyset.$ Then by continuity of $f$, we obtain $f(C_1)=f(C_2)$. Since such $C_1$ and $C_2$ are general, there exists an open set $V$ in the Shilov boundary of $\Omega$ such that $f$ is constant on $V$. Let $x\in V$.
Choose a totally geodesic polydisc $\Delta^{r}\subset \Omega$ such that $x\in (\partial\Delta)^{r}$. Then $V\cap (\partial\Delta)^{r}$ is open in $(\partial\Delta)^{r}$. We may assume $x=(1,\ldots,1)\in (\partial\Delta)^r$. Choose 
an $\epsilon_0>0$ such that $y=(e^{i\theta_1},\ldots, e^{i\theta_{r}})\in V\cap(\partial\Delta)^r$ for all $\theta_i\in (-\epsilon_0,\,\epsilon_0)$.
Choose geodesics $\gamma_i$ in $\Delta$ such that $1=\gamma_i(\infty)$ and $e^{i\theta_i}=\gamma_i(-\infty)$.
Then $\gamma:=(\gamma_1,\ldots,\gamma_{r})$ is a geodesic in $\Delta^{r}$ such that $x=\gamma(\infty)$ and $y=\gamma(-\infty)$
and therefore $f\circ\gamma$ is a geodesic in $\Omega'$ joining $f(x)=f\circ \gamma(\infty)$ and $f(y)=f\circ\gamma(-\infty)$.
But Lemma \ref{polyD} shows that $f\circ\gamma(\infty)$ and $f\circ\gamma(-\infty)$ should be contained in different boundary components, contradicting the assumption. 
Therefore for general  boundary component $C$ with rank$(C)>0$, $f(C)$ is not a point. Moreover, since $f$ preserves the boundary stratification, we obtain that for any general boundary component $C$, $f(C)$ is contained in a boundary component with rank greater or equal to the rank of $C$.

Let $C$ be a rank $(r-1)$ characteristic subdomain of $\Omega$ and let $\Delta_C$ be a minimal disc such that $\Delta_C\times C$ is totally geodesic in $\Omega$. Assume that $(0,0)\in \Delta_C\times C.$ Choose a unit speed complete geodesic $\gamma$ in $\Delta_C\times \{0\}$ connecting $(0,0)$ and a point $(p,0)\in \partial\Delta_C\times C$. Then there exists a polydisc $Q_p$ passing through $0$ that satisfies the condition in Lemma~\ref{polyD} with respect to $f(\gamma(\cdot), 0)$. By the condition \eqref{vector-rank}, we obtain
$\dim Q_p\leq r$. Moreover, similar to the proof of Lemma~\ref{component_totally}, we obtain
$$f(\gamma(t), x)\subset Q_p\times Q_p^\perp,\quad x\in C.$$

Define $F=(\Phi, \Psi)\colon\mathbb R\times C\to Q_p\times Q_p^\perp$ by 
$$F(t, x)=f(\gamma(t), x).$$ 
We may assume that $Q_p$ is the maximal polydisc such that each component of $\Phi$ is a unit speed geodesic in a disc. 
Since
$F$ is an isometry, as in the proof of Lemma~\ref{component_totally}, we obtain 
$$
	\lim_{t\to\infty} \Phi(t,x) = \lim_{t\to\infty} \Phi(t,y) \quad \text{ and } \lim_{t\to -\infty} \Phi(t,x) = \lim_{t\to -\infty} \Phi(t,y)
$$
and hence 
$$\Phi(\mathbb R,x) = \Phi(\mathbb R,y)\quad x, y\in C,$$
i.e. the image $\Phi(\mathbb R\times C)$ is real one-dimensional. Since $F$ is a totally geodesic isometry, this implies that $\Psi(t,\cdot):C\to Q_p^\perp$ is a totally geodesic isometry for all $t\in \mathbb  R.$ Furthermore, since $\partial_t \Phi\neq 0$, by condition \eqref{vector-rank}, we obtain that for any $s>0$,
$${\rm rank}~\Psi_*(s\partial/\partial t+v)\leq  r-\dim Q_p\leq r-1$$ 
for any $v\in T_x  C$.
Therefore by 
continuity of the derivatives of $f$, we obtain
$${\rm rank}~\Psi_*(v)\leq  r-1$$
for any $v\in T_x  C$.
By
induction argument, we may assume that $\Psi(t,\cdot)$ is a standard holomorphic embedding on $ C$ for all $t$ and $\Psi(t, \cdot)$ maps rank one vector to rank one vector.

Choose another rank one complete geodesic $\widetilde \gamma\subset \Delta_C$ passing through $0$. Then by the same argument, we obtain
$f(\widetilde\gamma(\mathbb R)\times C)\subset Q_{\widetilde p}\times Q_{\widetilde p}^\perp$, where $\widetilde p=\lim_{t\to 1}\widetilde \gamma(t)$.
Then either $Q_p\cap Q_{\tilde p}=\{0\}$ for some $\widetilde\gamma$ or $Q_p=Q_{\tilde p}$ for all $\tilde\gamma$. In the first case, we obtain $$f(\{0\}\times C)=(\Phi(0, C), \Psi(0, C))\subset \{0\}\times \left(Q_p^\perp\cap Q_{\widetilde p}^\perp\right).$$
Therefore any rank one vector $v\in TC$ is mapped to rank one vector. 
In the second case, we obtain 
$$f(\Delta_C\times C)\subset Q_p\times Q_p^\perp.$$ 
Write $f=(\Phi, \Psi):\Delta_C\times C\to Q_p\times Q_p^\perp.$
Since  $\Phi(\mathbb R,C)$ is real one-dimensional and $\Psi(t, \cdot)$ is a standard map on each polydisc, as in the proof of Lemma~\ref{component_totally}, we obtain that $\Psi$ is independent of $\zeta\in \Delta_C$.
Therefore $\Phi(\cdot, x):\Delta_C\to Q_p$ is a totally geodesic isometry for all $x\in C$.
Choose a complete geodesic $\gamma_1$ in $\Delta_C$. Then by the same argument we obtain
$$
	\Phi(\gamma_1(\mathbb R), x)=\Phi(\gamma_1(\mathbb R), y),\quad x, y\in C.
$$
Since $\gamma_1$ is arbitrary and $\Phi(\cdot, x)$ is an isometry, $\Phi$ is independent of $x\in C$. Therefore $f$ maps any rank one vector in $TC$ to rank one vector. 
For any $[v]\in \mathcal C(\Omega)$, there exists rank $(r-1)$ characteristic subdomain $C$ such that $v\in TC$. Hence $f$ preserves the rank one vectors. Then by Lemma~\ref{holo-disc} and Proposition~\ref{1-disc} we can complete the proof.
\end{proof}


\begin{thebibliography} {XXX}

\bibitem[A17]{Antonakoudis_2017}
Antonakoudis, Stergios M. {\it Isometric disks are holomorphic}. 
Invent. Math.  {\bf 207} (2017),  no. 3, 1289--1299.

\bibitem[CH04]{Choe_Hong_2004}
 Choe, Insong; Hong, Jaehyun 
{\it Integral varieties of the canonical cone structure on $G/P$}.
Math. Ann. {\bf 329} (2004), no. 4, 629--652. 



\bibitem[D78]{Drucker_1978}
Drucker, Daniel 
{\it Exceptional Lie algebras and the structure of Hermitian symmetric spaces}. 
Mem. Amer. Math. Soc. {\bf 16} (1978), no. 208, iv+207 pp.


\bibitem[EKK94]{Earle_Kra_Krushkal_1994}
Earle, C.J., Kra, I., Krushkal, S.L. 
{\it Holomorphic motions and Teichm\"uller spaces}. 
Trans.
Am. Math. Soc. {\bf 343}, 927--948 (1994)

\bibitem[GS13]{Gaussier_Seshadri_2013}
Gaussier, Herv\'e; Seshadri, Harish 
{\it Totally geodesic discs in strongly convex domains.} Math. Z.  {\bf 274} (2013),  no. 1--2, 185--197. 

\bibitem[GK82]{Greene_Krantz_1982}
 Greene, Robert E.; Krantz, Steven G. {\it Deformation of complex structures, estimates for the $\bar\partial$-equation, and stability of the  kernel}. Adv. in Math. {\bf 43} (1982), no. 1, 1--86. 

\bibitem[GH79]{Griffiths_Harris_1979}
Griffiths, Phillip; Harris, Joseph
{\it Algebraic geometry and local differential geometry}.
Ann. Sci. \'Ecole Norm. Sup. (4) {\bf 12} (1979), no. 3, 355--452.


\bibitem[K67]{Kobayashi}
Kobayashi, S. 
{\it Invariant distances on   complex manifolds   and holomorphic    mappings.} J. Math. Soc. Japan {\bf 19} (1967), 460--480.


\bibitem[L81]{Lempert_1981}
Lempert, L. 
{\it La m\'etrique de Kobayashi et la repr\'esentation des domaines sur la
boule.}
Bull. Soc. Math. France 109;4 (1981), 427--474. 

\bibitem[L82]{Lempert_1982}
 Lempert, L. 
{\it Holomorphic retracts and intrinsic metrics in convex domains}.
 Anal. Math. {\bf 8} (1982), no. 4, 257--261.

\bibitem[M89]{Mok89}
 Mok, Ngaiming 
{\it Metric rigidity theorems on Hermitian locally symmetric manifolds.} 
Series in Pure Mathematics, 6. World Scientific Publishing Co., Inc., Teaneck, NJ, 1989. xiv+278 pp. ISBN: 9971-50-800-1; 9971-50-802-8


\bibitem[M16]{Mok_2016}
 Mok, Ngaiming 
{\it Holomorphic isometries of the complex unit ball into irreducible bounded symmetric domains.} 
Proc. Amer. Math. Soc.  {\bf 144}  (2016),  no. 10, 4515--4525. 




\bibitem[MT92]{Mok_Tsai_1992}
Mok, Ngaiming; Tsai, I Hsun 
{\it Rigidity of convex realizations of irreducible bounded symmetric domains of rank $\geq 2$.} 
J. Reine Angew. Math. {\bf 431} (1992), 91--122. 



\bibitem[R08]{Roos_2008} 
G. Roos, 
{\it Exceptional symmetric domains}.
Symmetries in complex analysis, 157--189, Contemp. Math., {\bf 468}, Amer. Math. Soc., Providence, RI, 2008.

\bibitem[R71]{Royden_1971}
Royden, H. L.  
{\it Remarks on the Kobayashi metric}. 
Several complex variables, II (Proc. Internat. Conf., Univ. Maryland, College Park, Md., 1970), pp. 125--137. Lecture Notes in Math., Vol. {\bf 185}, Springer, Berlin, 1971. 


\bibitem[SV06]{Seshadri_Verma_2006}
 Seshadri, Harish; Verma, Kaushal 
{\it On isometries of the Carath\'eodory and Kobayashi metrics on strongly pseudoconvex domains.} 
Ann. Sc. Norm. Super. Pisa Cl. Sci. (5) {\bf 5} (2006), no. 3, 393--417.

\bibitem[S18]{Seo_2018}
Seo, Aeryeong 
{\it Remark on proper holomorphic maps between reducible bounded symmetric domains}. 
Taiwanese J. Math. 22 (2018), no. 2, 325--337.

\bibitem[T93]{Tsai_1993}  
 I. H. Tsai
 {\it Rigidity of proper holomorphic maps between symmetric domains},
 J. Differential Geom. {\bf 37} (1993),  no. 1, 123--160.
\end{thebibliography}
\end{document}